\documentclass[a4paper, 10pt]{article}
\mathcode`\<="4268     
\mathcode`\>="5269     
\mathcode`\:="603A     
\mathchardef\gt="313E  
\mathchardef\lt="313C  
\usepackage[english]{babel}
\usepackage[all]{xy}
\usepackage{amsthm}
\theoremstyle{definition}
\newtheorem{theorem}{Theorem}[section]

\newtheorem{proposition}[theorem]{Proposition}

\newtheorem{definition}[theorem]{Definition}
\newtheorem{remark}[theorem]{Remark}
\newtheorem{notation}[theorem]{Notation}

\DeclareMathAlphabet{\mathpzc}{OT1}{pzc}{m}{it}
\usepackage{bbm}
\usepackage{latexsym}
\def\span#1#2#3#4#5{\ensuremath{\xymatrix{#1&#3\ar[r]^{#5}\ar[l]_{#4}&#2}}}

\title{\textbf{Concrete fibrations}}
\author{Ruggero Pagnan\\
DISI, University of Genova\\
\texttt{ruggero.pagnan@disi.unige.it}}
\date{}
\begin{document}
\maketitle
\begin{abstract}
As far as we know, no notion of concreteness for fibrations exists. 
We introduce such a notion and discuss some basic results about it.
\end{abstract}
\def\a{\ensuremath{(a)}}
\def\adj#1#2{\ensuremath{\xymatrix@C.2cm{#1\ar@{-|}[r]&#2}}}
%
\def\adjpair#1#2#3#4#5{\ensuremath{\xymatrix{#1\ar@/^/[r]^{#3}_{\hole}="1"&#2\ar@/^/[l]^{#4}_{\hole}="2" \ar@{} "2";"1"|(.3){#5}}}}
%
\def\and{\ensuremath{\wedge}}
%
\def\arr{\ensuremath{\rightarrow}}
%
\def\Arr{\ensuremath{\Rightarrow}}
%
%
%
\def\arrow#1{\ensuremath{\cateb{#1}^{\arr}}}
%
\def\as{\ensuremath{\ast}}
%
%
\def\b{\ensuremath{\bullet}}
%
\def\BC{\ensuremath{(BC)}}
%
\def\BCd{\ensuremath{(BC)_d}}
\def\beps{\ensuremath{\backepsilon}}
%
\def\bicat#1{\ensuremath{\mathcal{#1}}}
%
\def\bf#1{\ensuremath{\mathbf{#1}}}
%
\def\bpararrow#1#2#3#4{\ensuremath{\xymatrix{#1\ar@<.9ex>[rr]^{#3}\ar@<-.9ex>[rr]_{#4}&&#2}}}
%
\def\bsquare#1#2#3#4#5#6#7#8{\xymatrix{#1\ar[dd]_{#8}\ar[rr]^#5&&#2\ar[dd]^{#6}\\
\\
#3\ar[rr]_{#7}&&#4}}
\def\bu{\ensuremath{\bullet}}
\def\btoind#1{#1\index{#1@\textbf{#1}}}
%
\def\card#1{\ensuremath{|#1|}}
\def\cate#1{\ensuremath{\mathcal{#1}}}
\def\cateb#1{\ensuremath{\mathbbm{#1}}}
%
\def\catebf#1{\ensuremath{\mathbf{#1}}}
\def\catepz#1{\ensuremath{\mathpzc{#1}}}
%
\def\cateop#1{\ensuremath{\cate#1^{op}}}
\def\comparrow#1#2#3#4#5#6{\ensuremath{\xymatrix{#6#1\ar[r]#4&#2\ar[r]#5&#3}}}
\def\cons#1{\ensuremath{\newdir{|>}{%
!/4.5pt/@{|}}\xymatrix@1@C=.3cm{\ar@{|=}[r]_{#1}&}}}
%
\def\cslice#1#2{\ensuremath{{\textstyle#1}/#2}}
\def\D{\ensuremath{(D)}}
\def\enunciato#1#2#3#4#5#6#7{\newtheorem{#4}#5{#1}#6                                              
\begin{#4}#2\label{#7}
#3
\end{#4}}
%
\def\dfn#1#2#3#4#5#6#7{\newtheorem{#4}#5{#1}#6                                              
\begin{#4}#2\label{#7}
\emph{#3}
\end{#4}}
%
\def\epiarrow#1#2#3#4{\ensuremath{\newdir{|>}{%
!/4.5pt/@{|}*:(1,-.2)@^{>}*:(1,+.2)@_{>}}\xymatrix{#3#1\ar@{-|>}[r]#4&#2}}}
\def\epitip{\ensuremath{\newdir{|>}{%
!/4.5pt/@{|}*:(1,-.2)@^{>}*:(1,+.2)@_{>}}}}
\def\eps#1{\ensuremath{\epsilon#1}}
\def\esempio#1#2#3#4#5#6#7{\newtheorem{#4}#5{#1}#6                                              
\begin{#4}#2\label{#7}
\textup{#3}
\begin{flushright}
$\blacksquare$
\end{flushright}
\end{#4}}
%
\def\et#1{\ensuremath{\eta#1}}
\def\Ex#1{\mbox{\boldmath\ensuremath{\exists}}{#1}}
\def\ex#1{\ensuremath{\exists#1}}
\def\F{\ensuremath{(F)}}
\def\fam#1{\ensuremath{\textup{Fam}(#1)}}
\def\farrow#1#2#3{\ensuremath{\xymatrix{#3:#1\ar[r]&#2}}}
\def\Fi{\ensuremath{\Phi}}
%
\def\fib#1#2#3{\ensuremath{
#3:\cateb{#1}\arr\cateb{#2}}}
%
%
\def\fibs#1#2#3#4{\ensuremath{\begin{array}{ll}
\slice{\catebf{#1}}{#2}\\
\,\downarrow^{\catebf{#4}}\\
\catebf{#3}
\end{array}}}
\def\fibss#1#2#3#4#5{\ensuremath{\begin{array}{ll}
\slice{\catebf{#1}}{#2}\\
\,\downarrow^{#5}\\
\slice{\catebf{#3}}{#4}
\end{array}}}
\def\forev#1{\ensuremath{\forall~#1,~\forall~}}
\def\fract#1{\ensuremath{\cateb{#1}[\Sigma^{-1}]}}
%
\def\G#1{\ensuremath{\int{#1}}}
%
%
\def\harrow{\ensuremath{\rightharpoonup}}
%
\def\hook{\ensuremath{\hookrightarrow}}
\def\hset#1#2#3{\ensuremath{\cate{#1}(#2,#3)}} 
\def\Im#1{\ensuremath{\mathit{Im}#1}}
\def\implies{\ensuremath{\Rightarrow}}
\def\indprod#1#2{\ensuremath{\prod#1#2}}
%
\def\infer#1#2{\ensuremath{\left.\begin{array}{cc}
#1\\
\xymatrix{\ar@{-}@<-.3ex>[rrrr]\ar@{-}@<.3ex>[rrrr]&&&&}\\
#2\\
\end{array}\right.}}
%
%
\def\it#1{\ensuremath{\mathit{#1}}}
%
\def\l#1{\ensuremath{#1_{\bot}}}
\def\larrow{\ensuremath{\leftarrow}}
\def\lcomparrow#1#2#3#4#5{\ensuremath{\xymatrix{#1\ar[rr]#4&&#2\ar[rr]#5&&#3}}}
\def\lfract#1{\ensuremath{[\Sigma^{-1}]\cateb{#1}}}
%
\def\lhookarrow#1#2#3{\ensuremath{\xymatrix{#1\ar@{^{(}->}[rr]^{#3}&&#2}}}
\def\mod#1#2{\ensuremath{\cateb{Mod}(#1,#2)}}
\def\Mono#1{\ensuremath{\mathit{Mono}#1}}
\def\monoarrow#1#2#3#4{\ensuremath{\newdir{
>}{{}*!/-5pt/@{>}}\xymatrix{#3#1\ar@{ >->}[r]#4&#2}}}
\def\monotail{\ensuremath{\newdir{
>}{{}*!/-5pt/@{>}}}}
\def\N{\ensuremath{(N)}}
\def\natarrow#1#2#3{\ensuremath{\newdir{|>}{%
!/4.5pt/@{|}*:(1,-.2)@^{>}*:(1,+.2)@_{>}}\xymatrix{#3:#1\ar@{=|>}[r]&#2}}}
\def\natbody{\ensuremath{\newdir{|>}{%
!/4.5pt/@{|}*:(1,-.2)@^{>}*:(1,+.2)@_{>}}}}
\makeatletter
\def\natpararrow#1#2#3#4#5{\@ifnextchar(
 {\Natpararrow{#1}{#2}{#3}{#4}{#5}}
 {\Natpararrow{#1}{#2}{#3}{#4}{#5}(.5)}}
\makeatother
\def\Natpararrow#1#2#3#4#5(#6){\ensuremath{\newdir{|>}{%
!/4.5pt/@{|}*:(1,-.2)@^{>}*:(1,+.2)@_{>}}
\xymatrix{#1\ar@<1.5ex>[rr]^(#6){#3}|(#6){\vrule height0pt depth.7ex
width0pt}="a"\ar@<-1.5ex>[rr]_(#6){#4}|(#6){\vrule height1.5ex width0pt}="b"&&#2
\ar@{=|>} "a";"b"#5}}}
\def\om{\ensuremath{\omega}}
%
\def\ov#1{\ensuremath{\overline{#1}}}
%
\def\ovarr#1{\ensuremath{\overrightarrow{#1}}}
%
\def\pararrow#1#2#3#4{\ensuremath{\xymatrix{#1\ar@<.8ex>[r]^{#3}\ar@<-.8ex>[r]_{#4}&#2}}}
\def\p#1{\ensuremath{\mathds{P}_{\catepz{S}}}(#1)}
\def\P#1{\ensuremath{\cateb{P}#1}}
\def\pr#1{\textbf{Proof:~}#1 
\begin{flushright}
$\Box$
\end{flushright}
}
%
%
%
%
\def\rel#1#2#3{\ensuremath{\xymatrix{#1:#2\ar[r]
|-{\SelectTips{cm}{}\object@{|}} &#3}}}
%
%
\def\res#1#2{\ensuremath{#1_{\rbag #2}}}
\def\S#1{\ensuremath{\catepz{S}_{\cateb{#1}}}}
\def\sectoind#1#2{#1\index{#2!#1}}
%
\def\sinfer#1#2{\ensuremath{\left.\begin{array}{cc}
#1\\
\xymatrix{\ar@{-}[rrrr]&&&&}\\
#2\\
\end{array}\right.}}
%
\def\slice#1#2{\ensuremath{#1/{\textstyle#2}}}
\def\sslice#1#2{\ensuremath{#1//{\textstyle#2}}}
\def\square#1#2#3#4#5#6#7#8{$$\xymatrix{#1\ar[d]_{#8}\ar[r]^{#5}&#2\ar[d]^{#6}\\
#3\ar[r]_{#7}&#4}$$}
\def\sub#1#2#3{\ensuremath{#1(#2#3)~}}
\def\summ#1{\ensuremath{\sum#1}}
%
\def\ti#1{\ensuremath{\tilde#1}}
%
\def\tens{\ensuremath{\otimes}}
%
\def\teo#1{\ensuremath{\tau#1}}
\def\ter#1{\ensuremath{\tau_{\cate{#1}}}}
\def\teta{\ensuremath{\theta}}
%
%
\def\toind#1{#1\index{#1}}
%
\def\un#1{\underline{#1}}
\def\vcell{\ensuremath{\ar@<-1ex>@{}[r]_{\hole}="a"\ar@<+1ex>@{}[r]^{\hole}="b"}}
%
\def\veps{\ensuremath{\varepsilon}}
%
\def\vfi{\ensuremath{\varphi}}
\def\w{\ensuremath{\wedge}}
\def\wti#1{\ensuremath{\widetilde#1}}
\def\y#1{\ensuremath{y#1}}

\section{Introduction}
A concrete category is one with a faithful functor to the category of sets and functions, in which case is also more specifically referred to as a construct. The idea is that a concrete category has to be thought of as to a catgory of sets equipped with an unspecified structure and homomorphisms between them.\\
As far as we know no notion of concreteness for fibrations exist. The main aim of this paper is that of discussing the topic in adherence with the point of view pursued in~\cite{MR780520}, together with some basic results about it.\\
In doing this for concreteness, the starting point was the recognition that the Isbell condition, definition~\ref{Isb}, is, in fibrational terms, nothing but a condition 
expressing a precise representability relationship between the whole fibration and its base, as it is typical for all the ``smallness'' conditions for fibrations. Think of local smallness or comprehension, for example. We adopted this approch with the intention of getting rid of any set-theoretic definition or characterization of concreteness. 
Coherently, one is prepared to find out that theorem~\ref{teo} no longer holds in general for fibrations, whereas a suitable version of it must be expected to hold, see theorem~\ref{daqui}. We end by discussing concreteness of small fibrations in some detail.

\section{Preliminaries}\label{prelim}
We assume that the reader is already acquainted with
ordinary category theory. In any case, she may
consult~\cite{MR1712872} for example, whereas
for further details concerning the content of the present section
she may consult~\cite{MR0156884},~\cite{MR0322006},~\cite{MR735079},~\cite{MR2240597},~\cite{MR1079899}.
Throughout the paper the metatheoretic framework is a theory of sets
and classes with the axiom of choice for classes.\\

Following~\cite{MR2240597}, with personal notation, we give the following
\begin{definition}
Let \cateb{B} be a category. A \emph{concrete category} over \cateb{B} is a pair 
$(\cateb{C},U)$ where \cateb{C} is a category and $U:\cateb{C}\arr\cateb{B}$ is a faithful functor. A concrete category over the category of sets and functions, \cateb{Sets}, will be henceforth referred to as a \emph{construct}. We say that a category \cateb{C} is concrete over \cateb{B}, or a construct, to implicitly mean that it is equipped with a faithful functor from it to \cateb{B}, or \cateb{Sets}, respectively.
\end{definition}
For every objects $A$, $B$ in a category \cateb{C}, a \emph{span} 
from $A$ to $B$ or $(A,B)$-\emph{span}, is a diagram $\span{A}{B}{X}{f}{g}$, whereas an
$(A,B)$-\emph{cospan} is an $(A,B)$-span in $\cateb{C}^{op}$. When no
confusion is likely to arise
spans and cospans will be also briefly written as triplets
$(f,X,g)$ and in both cases we will refer to $f$ and $g$ as to their \emph{components}.
Two $(A,B)$-spans $(f,X,g)$, $(f',X',g')$ are said to be \emph{equivalent} if for every
$(A,B)$-cospan $(h,Z,k)$, $hf=kg$ if and only if $hf'=kg'$, which fact
will be also more briefly indicated $(f,X,g)\sim(f',X',g')$.
\begin{definition}\label{Isb}
A category satisfies the \emph{Isbell condition} if for every
objects $A$ and $B$, there exists a set $\Sigma_{A,B}$ of $(A,B)$-spans no two different of which are equivalent and such that each $(A,B)$-span is equivalent to one
element in $\Sigma_{A,B}$. Henceforth, $\Sigma_{A,B}$ will be also
referred to as a \emph{choice set}.
\end{definition}
\begin{theorem}\label{teo}
A category \cateb{C} is a construct if and only if it satisfies the
Isbell condition.
\end{theorem}
\begin{proof}
See~\cite{MR0322006}, and~\cite{MR0399206} where the proof is carried out 
through the explicit construction of a faithful functor from \cateb{C} to \cateb{Sets}.
\end{proof}

\section{Concrete fibrations}\label{concretefib}
We here discuss concreteness of fibrations. 
Regarding the fibered category theory involved we only
recall the definition of fiber category, cartesian morphism and
fibered category, mostly with the intent of establishing terminology and
notation. We assume basic knowledge in fibered category theory as
well as knowledge of some advaced topics in it.
We refer the reader to~\cite{Grothendieck},~\cite{MR780520},
~\cite{Streicher},~\cite{MR1674451} and~\cite{MR1313497}. 
Nonetheless, some of the relevant notions will be briefly recalled when needed.
\begin{definition}
Let \fib{X}{B}{P} be a functor.
For every object $I$ in \cateb{B}, a morphism in \cateb{X} is said to
be \emph{vertical} over $I$ with respect to $P$ or 
$P$-\emph{vertical} over $I$ if its image under $P$ is the identical morphism at $I$.
The $P$-vertical morphisms at $I$ identify
a category $\cateb{X}_I$ which will be henceforth
referred to as \emph{fiber category} over $I$ with respect to $P$ or
the $P$-\emph{fiber} over $I$.
A morphism $\vfi:X\arr Y$ in \cateb{X} is said to be \emph{cartesian} with
respect to $P$ or $P$-\emph{cartesian} if for every morphisms
$v:J\arr PX$, $g:Z\arr Y$, with $Pg=P\vfi\circ v$, there exists a
unique morphism $\gamma:Z\arr X$ with $\vfi\circ\gamma=g$ and
$P\gamma=v$. The functor $P$ is said to be a \emph{fibration} or a
\emph{fibered category} if for every object $Y$ in \cateb{X} and for
every morphism $u:I\arr PY$, there exists an
$P$-cartesian morphism $\vfi:X \arr Y$ with $P\vfi =u$, to which we
will henceforth refer to as $P$-\emph{cartesian lifting} or
$P$-\emph{reindexing} of $Y$ along $u$. The categories \cateb{X} and \cateb{B}
will be also respectively referred to as \emph{total category} and
\emph{base category} of the fibration $P$.
\end{definition}

We recall that a fibration is said to be \emph{cloven}
if it comes equipped with chosen $P$-reindexings.
If \fib{X}{B}{P} is a cloven  fibration, then
for every object $Y$ in \cateb{X} and morphism $u:I\arr PY$ in \cateb{B}
a chosen $P$-reindexing of $Y$ along $u$ will be usually written
$\ov{u}Y:u^*Y\arr Y$. Fibered functors among cloven fibrations are not
assumed to preserve the chosen reindexings.\\

For every category with pullbacks \cateb{B}, the codomain functor $\catebf{cod}:\arrow{\cateb{B}}\arr\cateb{B}$ is a fibration which is usually referred to as \emph{fundamental fibration} over \cateb{B} since it allows \cateb{B} to be fibered over itself.\\

Following~\cite{MR0393180}, with personal terminology and notation, we give the following
\begin{definition}
A \emph{category with small morphisms} is a pair $(\cateb{B}, \cate{S})$ where \cateb{B} is category with pullbacks and \cate{S} is a class of morphisms of \cateb{B}, referred to as \emph{small}, satisfying the following requirements:
\begin{itemize}
\item[-] every isomorphism is small.
\item[-] small morphisms are closed under composition.
\item[-] small morphisms are stable under pullback along any morphism of \cateb{B}.
\item[-] if $f$ and $f\circ g$, whenerver the compostite makes sense, are small, then $g$ is small.
\end{itemize}
If \cateb{B} has a terminal object, then a \emph{small object} is one whose unique morphism to it is small. We will say that \cateb{B} is a category with small morphisms to implicitly mean that it is a category with pullbacks together with a class of small morphisms. Whenever we won't explicitly or implicitly refer to any notion of smallness in a fixed category, it has to be understood that each morphism is considered as small.
\end{definition}

\begin{remark}
Categories with small morphisms should intuitively be thought of as categories of classes equipped with a suitable notion of smallness allowing a class/set distinction. Small morphisms have to be thought of as morphisms whose fibers are small objects. In the framework or Algebraic Set Theory they are usually required to satisfy a certain amount of axioms to provide theories of sets and classes as flexible as possible. The reader may further consult~\cite{MR1368403}.
\end{remark}

\begin{definition}
Let \cateb{B} be a category with small morphisms.
A \emph{concrete fibration} over \cateb{B} is a pair $(P,U)$ where \fib{X}{B}{P} is a fibration and $U:P\arr\catebf{cod}$ is a faithful fibered functor over \cateb{B} such that for every morphism $f:X\arr Y$  in \cateb{X}, $Uf$ is a commuting square 
\[\xymatrix{|X|\ar[d]_{UX}\ar[r]^{|f|}&|Y|\ar[d]^{UY}\\
PX\ar[r]_{Pf}&PY}\]
in which $UX$ and $UY$ are small morphisms. A concrete fibration over \cateb{Sets} will be henceforth referred to as a \emph{fibered construct}.
\end{definition}
\begin{definition}
Let \fib{X}{B}{P} be a functor and $A$, $B$ be objects in the same
$P$-fiber. An $(A,B)$-span in $P$ is an $(A,B)$-span $(f,X,g)$ in \cateb{X}
with $Pf=Pg$. A vertical $(A,B)$-span in $P$ is an $(A,B)$-span in
$P$ whose components are $P$-vertical morphisms.
An $(A,B)$-cospan in $P$ is an $(A,B)$-span in
the opposite functor $P^{op}$, and
a vertical $(A,B)$-cospan is an $(A,B)$-cospan whose components are
$P^{op}$-vertical morphisms. Two $(A,B)$-spans $(f,X,g)$,
$(f',X',g')$ are said to be $P$-\emph{equivalent} if $Pf=Pf'$
and for every $(A,B)$-cospan $(h,Z,k)$, $hf=kg$ if and only if
$hf'=kg'$, which fact will be also indicated
$(f,X,g)\sim_P(f',X',g')$.
\end{definition}

\begin{definition}\label{concretefibration}

Let \cateb{B} be a category with small morphisms and \fib{X}{B}{P} be a fibration.
$P$ satisfies the \emph{Isbell condition} 
if for every objects $A$, $B$ in the same
$P$-fiber, there exists an $(A,B)$-span $(\pi_A,R,\pi_B)$ in $P$, with $P\pi_A$ a small morphism,
such that for every $(A,B)$-span $(f, X, g)$, there exists a unique $P$-cartesian morphism
$\theta:S\arr R$, such that $(f,X,g)\sim_P(\pi_A\theta,S,\pi_B\theta)$. 
The span $(\pi_A,R,\pi_B)$ will be also referred to as a \emph{choice span} for $P$.
\end{definition}

\begin{proposition}\label{clovconcr}
Let \cateb{B} be a category with small morphisms and \fib{X}{B}{P} be a cloven fibration.
$P$ satisfies the Isbell condition if and only
if for every object $I$ in \cateb{B}, and objects $A$, $B$ in
$\cateb{X}_I$, there exists a small morphism $\pi:\Sigma_{A,B}\arr I$
together with an $P$-vertical $(\pi^*A,\pi^*B)$-span $(p_A,R,p_B)$
such that for every morphism $u:J\arr I$ and $P$-vertical
$(u^*A,u^*B)$-span $(a,X,b)$ there exists a unique morphism $\ov{u}:J\arr\Sigma_{A,B}$ such that
$\pi\circ\ov{u}=u$ and for every morphism $v:I\arr K$ and
$(A,B)$-cospan $(h,Z,k)$ with $Ph=Pk=v$,
the outer part of diagram

\[\xymatrix@C=8ex{&&u^*A\ar[d]^{\ti{h}}\\
  X\ar@/^1pc/[urr]^a\ar@/_1pc/[drr]_b&
\ov{u}^*R\ar@/^/[ur]^(.4){\ti{p}_A}\ar@/_/[dr]_(.4){\ti{p}_B}&u^*v^*Z\\
&&u^*B\ar[u]_{\ti{k}}}\]
commutes over $J$ if and only if its inner part does, where the
$P$-vertical morphisms $\ti{p}_A$, $\ti{p}_B$, $\ti{h}$, $\ti{k}$
have been uniquely obtained as shown in the following
diagrams:
\[\xymatrix{u^*A\ar[r]&\pi^*A\\
\ov{u}^*R\ar@{-->}[u]^{\ti{p}_A}
\ar@{-->}[d]_{\ti{p}_B}\ar[r]^{\ti{u}R}&R\ar[u]_{p_A}\ar[d]^{p_B}\\
u^*B\ar[r]&\pi^*B}
\qquad
\xymatrix{u^*A\ar@{-->}[d]_{\ti{h}}\ar[r]^{\ov{u}A}&A\ar[dr]^h\\
u^*v^*Z\ar[rr]^{\ov{v\circ u}Z}&&Z\\
u^*B\ar@{-->}[u]^{\ti{k}}\ar[r]_{\ov{u}B}&B\ar[ur]_k}\]
where $\ti{u}R$ is a chosen $P$-reindexing of $R$ along \ov{u} and
the upper and lower horizontal morphisms in the lefmost diagram are
uniquely induced and $P$-cartesian over \ov{u}.

\end{proposition}

\begin{proof}
Straightforward.
\end{proof}

\begin{remark}\label{aquesto}
Referring to definition~\ref{concretefibration}, it is worth 
briefly discussing the case in which the $P$-cartesian morphism $\theta$ is a mediating one, that is when in fact $\theta:X\arr R$, $\pi_A\theta=f$ and $\pi_B\theta=g$, and in turn  the remaining part of the required universal property is automatically verified. 
With respect to a cloven fibration $P$ over a base category with small morphisms \cateb{B} the satisfaction of the Isbell condition in this special case, amounts to the following: for every object $I$ in \cateb{B}, and objects $A$, $B$ over $I$, 
there exists a small morphism $\pi:\Sigma_{A,B}\arr I$
together with an $P$-vertical $(\pi^*A,\pi^*B)$-span $(p_A,R,p_B)$
such that for every morphism $u:J\arr I$ and $P$-vertical
$(u^*A,u^*B)$-span $(a,X,b)$, there exists a unique morphism
$\ov{u}:J\arr\Sigma_{A,B}$ such that $\pi\circ\ov{u}=u$, and
$\ov{u}^*(p_A)\simeq a$, $\ov{u}^
*(p_B)\simeq b$, vertically over $J$.
\end{remark}

\begin{proposition}
Let \fib{X}{B}{P} be a fibration that satisfies the Isbell condition. The following facts hold:
\begin{itemize}
\item[(i)] for every object $X$ in \cateb{X}, the assignment $X\mapsto\Sigma_{X,X}\arr PX$ extends to a fibered funtor $P\arr \catebf{cod}$ over \cateb{B}.
\item[(ii)] if moreover $(\cateb{X},P)$ is a concrete category over \cateb{B}, then the previous assignment extends to a faithful fibered functor. 
\end{itemize}
\end{proposition}
\begin{proof}
Straightforward.
\end{proof}

\begin{proposition}\label{daqui}
Let \cateb{C} be a category. The following facts are equivalent:
\begin{itemize}
\item[(i)] \cateb{C} is a construct.
\item[(ii)] \cateb{C} satisfies the Isbell condition.
\item[(iii)] $\textbf{proj}:\textup{Fam}(\cateb{C})\arr\cateb{Sets}$ is a fibered construct.
\item[(iv)] $\textbf{proj}:\textup{Fam}(\cateb{C})\arr\cateb{Sets}$ satisfies the Isbell condition.
\end{itemize}
\end{proposition}
\begin{proof}
(i)$\Leftrightarrow$(ii): see theorem~\ref{teo}. (i)$\Leftrightarrow$(iii):
every faithful functor $\cateb{C}\arr\cateb{Sets}$ extends to a faithful fibered functor as required in (iii) and, viceversa, every faithful fibered functor as in (iii) restricts to a 
faithful functor as in (i).
(ii)$\Arr$(iv): let $I$ be a set and 
$A\doteq(A_i)_{i\in I}$,
$B\doteq(B_i)_{i\in I}$ be $I$-indexed families of objects of
\cateb{C}. Since
\cateb{C} satisfies the Isbell condition,
for every $i\in I$ there exists a
choice set $\Sigma_{A_i,B_i}$ for the $(A_i,B_i)$-spans in
\cateb{C}. Let its elements be indicated as 
$(\pi,R_{(\pi,\pi')}^i,\pi')$. Put $\Sigma_{A,B}\doteq\bigsqcup_{i\in
  I}\Sigma_{A_i,B_i}$ and let $p:\Sigma_{A,B}\arr I$ be the evident
projection. We claim that the $(A,B)$-span
\[\xymatrix@R=3ex{&(A_i)_{i\in I}\\
(R_{(\pi,\pi')}^i)_{(i,(\pi,\pi'))\in\Sigma_{A,B}}
\ar@/^/[ur]^{(p,\pi_A)}\ar@/_/[dr]_{(p,\pi_B)}\\
&(B_i)_{i\in I}}\]
with $\pi_A=(\pi:R_{(\pi,\pi')}^i\arr
A_i)_{(i,(\pi,\pi'))\in\Sigma_{A,B}}$,
$\pi_B=(\pi':R_{(\pi,\pi')}^i\arr
B_i)_{(i,(\pi,\pi'))\in\Sigma_{A,B}}$ is a choice span 
for \catebf{proj}. Indeed,
for every $(A,B)$-span $((u,f),(X_j)_{j\in J},(u,g))$ with $u:J\arr I$, $f=(f_j:X_j\arr
A_{u(j)})_{j\in J}$, $g=(X_j\arr B_{u(j)})_{j\in J}$,
let $\ov{u}:J\arr\Sigma_{A,B}$ be the function  $\ov{u}(j)\doteq(u(j),(\pi_{f_j},R_{(\pi_{f_j},\pi'_{g_j})}^{u(j)},\pi'_{g_j}))$, with
$(\pi_{f_j},R_{(\pi_{f_j},\pi'_{g_j})}^{u(j)},\pi'_{g_j})\in\Sigma_{A_{u(j)},B_{u(j)}}$ 
the unique span equivalent to $(f_j,X_j,g_j)$ in \cateb{C}.
Put 
\[\xymatrix{\theta\doteq(\ov{u},id):(R_{(\pi_{f_j},\pi'_{g_j})}^{u(j)})_{j\in
  J}\ar[r]&(R_{(\pi,\pi')}^i)_{(i,(\pi,\pi'))\in\Sigma_{A,B}}}\]
and observe that for every
$(A,B)$-cospan $((v,h),(Z_l)_{l\in L},(v,k))$, with $v:I\arr L$,
$h=(A_i\arr Z_{v(i)})_{i\in I}$, $k=(B_i\arr Z_{v(i)})_{i\in I}$,
and for every $j\in J$,
$h_{u(j)}\circ f_j=k_{u(j)}\circ g_j$ if and only if
$h_{u(j)}\circ\pi_{f_j}=k_{u(j)}\circ\pi_{g_j}$, since 
\cateb{C} satisfies the Isbell condition, 
and that $\theta$ is unique by construction. (iv)\Arr(i): \cateb{C} satisfies the Isbell condition because \textbf{proj} satisfies the Isbell condition with respect to
$1$-indexed families of objects of \cateb{C} in particular, with $1$ a terminal
object in \cateb{Sets}.
\end{proof}

\section{Concreteness of small fibrations}
The notion of internal category makes sense with respect to an ambient category with pullbacks. Unless otherwise explicitly stated, we will henceforth assume such a minimal requirement. We do not fully recall what an internal category is. We describe it as a $6$-tuple $\catebf{C}=(C_0,C_1,d_0,d_1,c,i)$ where
$d_0,d_1:C_1\arr C_0$, $c:C_1\times_{C_0}C_1\arr C_1$, $i:C_0\arr C_1$ fit in suitable commutative diagrams, with
\[\xymatrix{C_1\times_{C_0}C_1\ar[d]_{\pi_1}\ar[r]^(.6){\pi_2}&C_1\ar[d]^{d_1}\\
C_1\ar[r]_{d_0}&C_0}\]
a pullback. A full definition can be found
in~\cite{MR1674451},~\cite{MR0470019}. We will refer to an internal category in \cateb{Sets} as to a \emph{small category}.

\begin{definition}
Let $\catebf{C}=(C_0,C_1,d_0,d_1,c,i)$ 
be an internal category in a category \cateb{B}. An \emph{internal diagram} on \catebf{C} is a pair $(p:F\arr C_0,q:C_1\times_{C_0}F\arr F)$ where 
\[\xymatrix{C_1\times_{C_0}F\ar[d]_{\pi_1}\ar[r]^(.6){\pi_2}&F\ar[d]^p\\
C_1\ar[r]_{d_0}&C_0}\]
is a pullback, and such that the following identities hold:
\begin{itemize}
\item[-] $p\circ q=d_1\circ\pi_1$.
\item[-] $q\circ<ip,id_F>=id_F$.
\item[-] $q\circ(C_1\times_{C_0} q)=q\circ(c\times_{C_0} F)$.
\end{itemize}
Internal diagrams on \catebf{C} will be also denoted $(p,q):\catebf{C}\arr\cateb{B}$.  
\end{definition}
\begin{notation}
Following~\cite{MR1300636}, where
internal diagrams are referred to as category actions, for every pair of morphisms $(f:I\arr C_1,a:I\arr F)$ we write $f\cdot a$ for the composite $q\circ<f,a>$.
\end{notation}

\begin{definition}
Let \catebf{C} be an internal category in a category \cateb{B}. 
An internal diagram $(p,q):\catebf{C}\arr\cateb{B}$ is \emph{faithful} if for every $f,g:I\arr C_1$, with $d_0f=d_0g$, $d_1f=d_1g$, and for every $a:I\arr F$ with $pa=d_0f$, if $f\cdot a=g\cdot a$, then $f=g$. 
\end{definition}

\begin{remark}
Internal categories with a faithful internal diagram may be referred to as concrete internal category but every internal category \catebf{C} is concrete in this sense, since 
it can be verified that a faithful internal diagram on it is $(d_1:C_1\arr C_0,c:C_1\times_{C_0}C_1\arr C_1)$. This is nothing but the internalization of the well known result that every small category has a faithful functor to \cateb{Sets}, see~\cite{MR0013131}.
\end{remark}
Now, let \catebf{C} be an internal category in \cateb{B}
and $\textup{Fam}(\catebf{C})$ be the category
identified by the following data:
\begin{description}
\item[objects:] are pairs $(I,X)$ where $I$ is an object of
  \cateb{B} and $X:I\arr C_0$ is in \cateb{B}.
\item[morphisms:] are pairs $(u,f):(I,X)\arr (J,Y)$ where $u:I\arr
  J$ and $f:I\arr C_1$ in \cateb{B}, fitting together in the
commutative diagram
\[\xymatrix{&I\ar[d]^f\ar[dl]_X\ar[r]^u&J\ar[d]^Y\\
C_0&C_1\ar[r]_{d_1}\ar[l]^{d_0}&C_0}\]
\item[composition:] is given by the rule
\[\xymatrix{(I,X)\ar@/_1pc/[rr]_{(v\circ u,\,c\circ<g\circ u,f>)}
\ar[r]^{(u,f)}&(J,Y)\ar[r]^{(v,g)}&(K,Z)}\]
\item[identity:] the identity morphism at $(I,X)$ say, is
  $(id_I,i\circ X)$.
\end{description}
For every object $(I,X)$ of
$\textup{Fam}(\catebf{C})$, the assignment $(I,X)\mapsto I$ extends to
a functor $\textbf{proj}:\textup{Fam}(\catebf{C})\arr\cateb{B}$ which
is a fibration referred to as the \emph{externalization} of
\catebf{C}. For every morphism $u:I\arr J$ and object
$(J,Y)$, a \textbf{proj}-cartesian lifting of $(J,Y)$ along $u$ can be
taken as
\[\xymatrix@C=10ex{(I,Y\circ u)\ar[r]^(.5){(u,\,i\circ Y\circ u)}&(J,Y)}\]
A \emph{small fibration} is one which is essentially the externalization of some
internal category in its base.\\

Internal diagrams on an internal category \catebf{C} in a category \cateb{B} correspond to fibered functors from the externalization of \catebf{C} to the fundamental fibration over \cateb{B}, see~\cite{MR1674451}.
\begin{proposition}\label{faith}
Let \catebf{C} be an internal category in a category \cateb{B}. Faithful internal internal diagrams on \catebf{C} give rise to faithful fibered functors from the externalization of \catebf{C} to the fundamental fibration over \cateb{B}.
\end{proposition}
\begin{proof}
Straightforward.
\end{proof}

\begin{proposition}\label{con}
Every small fibration over a base category with finite limits satisfies the Isbell condition.
\end{proposition}
\begin{proof}
Let $\catebf{C}=(C_0,C_1,d_0,d_1,c,i)$ be an internal category in a
category \cateb{B} with finite limits and
$\catebf{proj}:\textup{Fam}(\catebf{C})\arr\cateb{B}$ be its externalization.
For every object $I$ in \cateb{B} and $I$-indexed families $A,B:I\arr
C_0$, construct the diagram
\begin{eqnarray}\label{tutto}
\xymatrix{\Sigma_{A,B}\ar[dd]_{\pi}\ar[r]^h&S\ar[d]^{<s_1,s_2>}\ar[r]^{\sigma}&C_0
\ar[d]^{\Delta_0}\\
&C_1\times C_1\ar[d]^{d_1\times d_1}\ar[r]_{d_0\times d_0}&C_0\times C_0\\
I\ar[r]_(.4){<A,B>}&C_0\times C_0}
\end{eqnarray}
in which all the quadrilaterals are pullbacks. We claim that the $\catebf{proj}$-vertical
$(\pi^*(I,A),\pi^*(I,B))$-span that fits in diagram
\[\xymatrix@C=10ex@R=4ex{(\Sigma_{A,B},A\circ\pi)\ar[r]^(.6){(\pi,i\circ A\circ\pi)}&(I,A)\\
(\Sigma_{A,B},\sigma\circ h)\ar[u]^{(id,s_1\circ
    h)}\ar[d]_{(id,s_2\circ h)}\\
(\Sigma_{A,B},B\circ\pi)\ar[r]_(.6){(\pi,i\circ B\circ\pi)}&(I,B)\\
\Sigma_{A,B}\ar[r]_{\pi}&I}\]
is universal as described in remark~\ref{aquesto}.
Indeed, for every morphism $u:J\arr I$ and for every $\catebf{proj}$-vertical
$(u^*(I,A),u^*(I,B))$-span
\[\xymatrix{(J,A\circ
  u)&(J,X)\ar[r]\ar[l]_(.4){(id,f)}\ar[r]^{(id,g)}&(J,B\circ u)}\]
there exists a unique morphism $\ti{u}:J\arr
S$ with $\sigma\circ \ti{u}=X$ and $<s_1,s_2>\circ \ti{u}=<f,g>$. Thus,
in turn there exists a unique morphism $\ov{u}:J\arr\Sigma_{A,B}$
with $\pi\circ\ov{u}=u$ and $h\circ\ov{u}=\ti{u}$.
It can be seen that diagram
\[\xymatrix@C=12ex{(J,A\circ u)
\ar@{-->}[r]^(.45){(\ov{u},i\circ A\circ \pi\circ\ov{u})}&(\Sigma_{A,B},A\circ\pi)\\
(J,X)\ar[u]^{(id,f)}\ar[d]_{(id,g)}\ar[r]^(.45){(\ov{u},i\circ\sigma\circ
h\circ\ov{u})}&(\Sigma_{A,B},\sigma\circ h)\ar[d]^{(id,s_2\circ h)}\ar[u]_{(id,s_1\circ h)}\\
(J,B\circ u)\ar@{-->}[r]_(.45){(\ov{u},i\circ B\circ\pi\circ\ov{u})}&(\Sigma_{A,B},B\circ \pi)\\
J\ar[r]_{\ov{u}}&\Sigma_{A,B}}\]
commutes over $\ov{u}$.
\end{proof}
\begin{remark}
We observe that proposition~\ref{con} can be restated with respect to a base category with small morphisms by further requiring every diagonal morphism to be small, as well as all the structure morphisms of an internal category in it.
\end{remark}
Recall from~\cite{MR0268247},~\cite{MR1173016} the notion of \emph{pullback of a parallel pair} of morphisms in a category with pullbacks: for every morphisms $h:T\arr A$ and $f,g:X\arr A$, the pullback of the pair $(f,g)$ along $h$ 
can be constructed as shown in diagram
\[\xymatrix@R=.5ex{&&P\ar[dr]\ar[dl]\\
&\b\ar[dr]\ar[dl]&&\b\ar[dr]\ar[dl]\\
T\ar[dr]_h&&X\ar[dr]_g\ar[dl]^f&&T\ar[dl]^h\\
&A&&A}\]
in which all the quadrilaterals are pullbacks.

\begin{proposition}
Let $\catebf{C}=(C_0, C_1, d_0,d_1, c, i)$ be an internal category in a category \cateb{B} with finite limits. The externalization of \catebf{C} is a concrete fibration. 
\end{proposition}
\begin{proof}
We prove the thesis by showing that 
for every object $(I,X)$, the assignment 
$(I,X)\mapsto\Sigma_{X,X}\arr I$
extends to a faithful fibered functor as required. In view of proposition~\ref{faith} we do this by describing the faithful internal diagram that gives rise to the wanted faithful fibered functor. Construct the pullback of the parallel pair $(d_0\times d_0,d_1\times d_1)$ against $\Delta_0:C_0\arr C_0\times C_0$, like this:
\begin{eqnarray}\label{first}
\xymatrix@R=.5ex{&&F\ar[dr]^{p_1}\ar[dl]_{p_2}\\
&S\ar[dr]|(.4){<s_1,s_2>}\ar[dl]_{\sigma}&&T\ar[dr]^{\tau}\ar[dl]|(.4){<t_1,t_2>}\\
C_0\ar[dr]_(.4){\Delta_0}&&C_1\times C_1\ar[dr]_(.4){d_1\times d_1}\ar[dl]^(.4){d_0\times d_0}&&C_0\ar[dl]^(.4){\Delta_0}\\
&C_0\times C_0&&C_0\times C_0}
\end{eqnarray}
Let the first component of the wanted internal diagram be $p\doteq\tau\circ p_1$. Construct the pullback $(\pi_1,C_1\times_{C_0}F,\pi_2)$ of $p$ against $d_0$. Construct the morphism $\veps:C_1\arr F$, uniquely induced in diagram~(\ref{first}) by the triple $(d_1,\Delta_1:C_1\arr C_1\times C_1,d_0)$ and observe that it is monic. Construct the pullback $(q_1,F\times_{C_0}F,q_2)$ of $p$ against $\sigma\circ p_2$ and let $\mu:C_1\times_{C_0}F\arr F\times_{C_0}F$ be the uniquely induced morphism in it, determined by the pair $(\veps\pi_1,\pi_2)$. Construct the pullback $(\pi_1\times\pi_1,(C_1\times_{C_0}C_1)\times(C_1\times_{C_0}C_1),\pi_2\times \pi_2)$ of $d_1\times d_1$ against $d_0\times d_0$ and let $\gamma:F\times_{C_0}F\arr(C_1\times_{C_0}C_1)\times(C_1\times_{C_0}C_1)$ be the uniquely induced morphism determined by the pair $(<s_1,s_2>p_2q_1,<t_1,t_2>p_1q_2)$. Construct the second component $q:C_1\times_{C_0}F\arr F$ of the wanted internal diagram as the uniquely induced morphism in~(\ref{first}) determined by the triple $(d_1\pi_1,(c\times c)\gamma\mu,\sigma p_2\pi_2)$. Thus, the identity $pq=d_1\pi_1$ holds by construction, and the remaining identities 
are easily checked to hold, so that $(p,q):\catebf{C}\arr\cateb{B}$ and it is faithful. Indeed, let $f,g:I\arr C_1$ with $d_0f=d_0g$ and $d_1f=d_1g$ be such that for every $a:I\arr F$, $f\cdot a=g\cdot a$. Thus, in particular, it can be easily checked that $\veps f=f\cdot(\veps i\circ d_0f)=g\cdot(\veps i\circ d_0g)=\veps g$, and use the fact that $\veps$ is monic.
\end{proof}

\bibliographystyle{plain}
\bibliography{BiblioTeX}

\begin{thebibliography}{10}

\bibitem{MR735079}
Ji{\v{r}}{\'{\i}} Ad{\'a}mek.
\newblock {\em Theory of mathematical structures}.
\newblock D. Reidel Publishing Co., Dordrecht, 1983.

\bibitem{MR2240597}
Ji{\v{r}}{\'{\i}} Ad{\'a}mek, Horst Herrlich, and George~E. Strecker.
\newblock Abstract and concrete categories: the joy of cats.
\newblock {\em Repr. Theory Appl. Categ.}, (17):1--507, 2006.
\newblock Reprint of the 1990 original [Wiley, New York; MR1051419].

\bibitem{MR0393180}
Jean B{\'e}nabou.
\newblock Th\'eories relatives \`a un corpus.
\newblock {\em C. R. Acad. Sci. Paris S\'er. A-B}, 281(20):Ai, A831--A834,
  1975.

\bibitem{MR780520}
Jean B{\'e}nabou.
\newblock Fibered categories and the foundations of naive category theory.
\newblock {\em J. Symbolic Logic}, 50(1):10--37, 1985.

\bibitem{MR1313497}
Francis Borceux.
\newblock {\em Handbook of categorical algebra. 2}, volume~51 of {\em
  Encyclopedia of Mathematics and its Applications}.
\newblock Cambridge University Press, Cambridge, 1994.
\newblock Categories and structures.

\bibitem{MR0013131}
Samuel Eilenberg and Saunders MacLane.
\newblock General theory of natural equivalences.
\newblock {\em Trans. Amer. Math. Soc.}, 58:231--294, 1945.

\bibitem{Grothendieck}
A.~Grothendieck.
\newblock Cat{\'e}gories fibr{\'e}es et descente ({E}xpos{\'e} vi).
\newblock {\em Lecture Notes in Mathematics}, (224):145--194, 1970.

\bibitem{MR0156884}
J.~R. Isbell.
\newblock Two set-theoretical theorems in categories.
\newblock {\em Fund. Math.}, 53:43--49, 1963.

\bibitem{MR0322006}
Freyd~Peter J.
\newblock Concreteness.
\newblock {\em J. Pure Appl. Algebra}, 3:171--191, 1973.

\bibitem{MR1674451}
Bart Jacobs.
\newblock {\em Categorical logic and type theory}, volume 141 of {\em Studies
  in Logic and the Foundations of Mathematics}.
\newblock North-Holland Publishing Co., Amsterdam, 1999.

\bibitem{MR0470019}
P.~T. Johnstone.
\newblock {\em Topos theory}.
\newblock Academic Press [Harcourt Brace Jovanovich Publishers], London, 1977.
\newblock London Mathematical Society Monographs, Vol. 10.

\bibitem{MR1368403}
A.~Joyal and I.~Moerdijk.
\newblock {\em Algebraic set theory}, volume 220 of {\em London Mathematical
  Society Lecture Note Series}.
\newblock Cambridge University Press, Cambridge, 1995.

\bibitem{MR1173016}
G.~M. Kelly.
\newblock A note on relations relative to a factorization system.
\newblock In {\em Category theory ({C}omo, 1990)}, volume 1488 of {\em Lecture
  Notes in Math.}, pages 249--261. Springer, Berlin, 1991.

\bibitem{MR0268247}
Aaron Klein.
\newblock Relations in categories.
\newblock {\em Illinois J. Math.}, 14:536--550, 1970.

\bibitem{MR1712872}
Saunders Mac~Lane.
\newblock {\em Categories for the working mathematician}, volume~5 of {\em
  Graduate Texts in Mathematics}.
\newblock Springer-Verlag, New York, second edition, 1998.

\bibitem{MR1300636}
Saunders Mac~Lane and Ieke Moerdijk.
\newblock {\em Sheaves in geometry and logic}.
\newblock Universitext. Springer-Verlag, New York, 1994.
\newblock A first introduction to topos theory, Corrected reprint of the 1992
  edition.

\bibitem{Streicher}
T.~Streicher.
\newblock Fibred categories {\`a} la {J}ean {B}{\'e}nabou.
\newblock {\em www.mathematik.tu-darmstadt.de/~streicher/}.

\bibitem{MR1079899}
Magdalena Velebilov{\'a}.
\newblock Small congruences and concreteness.
\newblock {\em Proc. Amer. Math. Soc.}, 115(1):13--18, 1992.

\bibitem{MR0399206}
Ji{\v{r}}i Vin{\'a}rek.
\newblock A new proof of the {F}reyd's theorem.
\newblock {\em J. Pure Appl. Algebra}, 8(1):1--4, 1976.

\end{thebibliography}
\end{document}